\documentclass[a4paper,12pt]{amsart}
\usepackage[english]{babel}
\usepackage{amsmath,amssymb,amsthm}

\usepackage[pdftex]{color}
\usepackage[bookmarks=true,hyperindex,pdftex,colorlinks, citecolor=blue,linkcolor=blue, urlcolor=blue]{hyperref}

\theoremstyle{plain}
\newtheorem{theorem}{Theorem}
\newtheorem{corollary}[theorem]{Corollary}
\newtheorem{prop}[theorem]{Proposition}
\newtheorem{lemma}[theorem]{Lemma}

\theoremstyle{definition}
\newtheorem{remark}[theorem]{Remark}
\newtheorem{example}[theorem]{Example}

 \DeclareMathOperator{\supp}{supp\,}

\newcommand{\R}{\mathbb{R}}
\newcommand{\N}{\mathbb{N}}
 \newcommand{\Q}{\mathbb{Q}}
 \newcommand{\eps}{\varepsilon}

\renewcommand{\leq}{\leqslant}
\renewcommand{\geq}{\geqslant}

 \DeclareMathOperator{\NA}{NA}

\begin{document}
\title{Some geometric properties of Read's space}

\author[Kadets]{Vladimir Kadets}
\author[L\'opez]{Gin\'es L\'opez-P\'erez}
\author[Mart\'{\i}n]{Miguel Mart\'{\i}n}

\address[Kadets]{School of Mathematics and Computer Sciences \\
V. N. Karazin Kharkiv National University \\ pl.~Svobody~4 \\
61022~Kharkiv \\ Ukraine
\newline
\href{http://orcid.org/0000-0002-5606-2679}{ORCID: \texttt{0000-0002-5606-2679} }
}
\email{vova1kadets@yahoo.com}

\address[L\'opez]{Departamento de An\'{a}lisis Matem\'{a}tico \\ Facultad de
 Ciencias \\ Universidad de Granada \\ 18071 Granada, Spain
\newline
\href{http://orcid.org/0000-0002-3689-1365}{ORCID: \texttt{0000-0002-3689-1365} }
 }
\email{glopezp@ugr.es}

\address[Mart\'{\i}n]{Departamento de An\'{a}lisis Matem\'{a}tico \\ Facultad de
 Ciencias \\ Universidad de Granada \\ 18071 Granada, Spain
\newline
\href{http://orcid.org/0000-0003-4502-798X}{ORCID: \texttt{0000-0003-4502-798X} }
 }
\email{mmartins@ugr.es}

\date{April 3rd, 2017; revised: June 2nd, 2017}


\thanks{The research of the first author is done in frames of Ukrainian Ministry of Science and Education Research Program 0115U000481, and it was partially done during his stay in the University of Granada which was supported by the Spanish MINECO/FEDER grant MTM2015-65020-P. The research of the second and third authors is partially supported by Spanish MINECO/FEDER grant MTM2015-65020-P}

\begin{abstract}
We study geometric properties of the Banach space $\mathcal{R}$
constructed recently by C.~Read \cite{Read} which does not contain
proximinal subspaces of finite codimension greater than or equal
to two. Concretely, we show that the bidual of $\mathcal{R}$ is
strictly convex, that the norm of the dual of $\mathcal{R}$ is rough, and that $\mathcal{R}$ is weakly locally uniformly rotund (but it is not locally uniformly rotund). Apart of the own interest of the results, they provide a simplification of the proof by M.~Rmoutil
\cite{Rmoutil} that the set of norm-attaining functionals over
$\mathcal{R}$ does not contain any linear subspace of dimension
greater than or equal to two. Note that if a Banach space $X$ contains proximinal subspaces of finite codimension at least two, then the set of norm-attaining functionals over $X$ contain two-dimensional linear subspaces of $X^*$. Our results also provides positive answer to the questions of whether the dual of $\mathcal{R}$ is smooth and of whether $\mathcal{R}$ is weakly locally uniformly rotund \cite{Rmoutil}. Finally, we present a renorming of Read's space which is smooth, whose dual is smooth, and such that its set of norm-attaining functionals does not contain any linear subspace of dimension greater than or equal to two, so the renormed space does not contain proximinal subspaces of finite codimension greater than or equal to two.
\end{abstract}

\maketitle

\section{Introduction}
In his recent brilliant manuscript \cite{Read}, the late Charles J.~Read
constructed an equivalent norm $|||\cdot|||$ on $c_0$ such that
the space $\mathcal R = (c_0, |||\cdot|||)$ answers negatively the
following open problem by Ivan Singer from 1974 \cite{Singer}:
\begin{itemize}
\item[(S)] Is it true that every Banach space contains a proximinal subspace of finite codimension greater than or equal to $2$?
\end{itemize}
Recall that a subset $Y$ of a Banach space $X$ is said to be \emph{proximinal} if for every $x \in X$ there is $y_0 \in Y$ such that $\|x - y_0\| = \inf\{\|x-y\|\colon y\in Y\}$.

Many times, when an interesting and non-trivial space is
constructed to produce a counterexample, it can be also used to
solve some other different problems. In the case of Read's space,
this didn't wait for too long: in \cite[Theorem 4.2]{Rmoutil},
Martin Rmoutil demonstrates that the space $\mathcal R$ also gives
a negative solution to the following problem by Gilles Godefroy
\cite[Problem III]{Godefroy}:
\begin{itemize}
\item[(G)] is it true that for every Banach space $X$ the set $\NA(X)  \subset X^*$ of norm attaining functionals contains a two-dimensional linear subspace?
\end{itemize}
Recall that an element $f$ of the dual $X^*$ of a Banach space $X$ is said to be \emph{norm attaining} if there is $x\in X$ with $\|x\|=1$ such that $\|f\|=|f(x)|$.

The utility of Read's space makes clear that it is interesting to
increase our knowledge of its geometry. The aim of this note is to
show that Read's space fulfills the following properties:
\begin{itemize}
\item[(a)] the bidual $\mathcal{R}^{**}$ of Read's space is strictly convex;
\item[(b)] therefore, $\mathcal{R}^*$ is smooth; and
\item[(c)] $\mathcal{R}$ is also strictly convex;
\item[(d)] moreover, $\mathcal{R}$ is weakly locally uniformly rotund (WLUR);
\item[(e)] the norm of $\mathcal{R}^*$ is rough, so it is not Fr\'{e}chet differentiable at any point;
\item[(f)] and the norm of $\mathcal{R}$ is not locally uniformly rotund (LUR);
\item[(g)] moreover, there is $\rho>0$ such that every weakly open subset of the unit ball of $\mathcal{R}$ has diameter greater than or equal to $\rho$.
\end{itemize}

The main point in Rmoutil's proof was the demonstration that for
several closed subspaces $Y$ of $\mathcal R$, the corresponding
quotient spaces $X/Y$ are strictly convex. Observe that it follows
from assertion (b) above that ALL quotient spaces $\mathcal R/Y$
are strictly convex (because their duals $Y^\bot$ are smooth).
This gives a substantial simplification of the proof of
\cite[Theorem 4.2]{Rmoutil}. Even more, it follows from a 1987
paper by V.~Indumathi \cite[Proposition~1]{Indumathi-JAT}, that if
$X$ is a Banach space with $X^*$ smooth at every point of
$\NA(X)\cap S_{X^*}$, then a finite-codimensional subspace $Y$ of
$X$ is proximinal if and only if $Y^\perp\subset \NA(X)$. This
hypothesis on $X$ is satisfied if $X^*$ is smooth (clear) or if
$X$ is WLUR (see \cite{Yorke}). Let us comment that the facts that $\mathcal{R}$ is WLUR and $\mathcal{R^*}$ is smooth were said to be unexpected in the cited paper  \cite{Rmoutil} by Rmoutil.

As a final result of the paper, we present a renorming of Read's space which is smooth, whose dual is smooth and which solves negatively both problems (S) and (G).

Let us now present the notation we are using along the paper. We deal only with real scalars and real Banach spaces. By $\|\cdot\|_1$ and $\|\cdot\|_\infty$ we denote the standard norms on $\ell_1$ and $\ell_\infty$ respectively, and by $(e_n)_{n\in \N}$ we denote the canonical basis vectors,  i.e.\  the $k$-th coordinate $e_{n,k}$  of $e_n$ equals 0 for $n \neq k$ and equals $1$ for $n = k$. In the sequel, each time it will be clear from the context in what sequence space these $e_n$'s are considered. For $x = (x_k)_{k \in \N} \in \ell_\infty$,  $y = (y_k)_{k \in \N} \in \ell_1$, we use the standard notation $\langle x, y \rangle =  \sum_{n \in \N} x_n y_n$. If $X$ is an arbitrary Banach space, $B_X$ denotes its closed unit ball, $S_X$ denotes its unit sphere, and $X^*$ is the dual of $X$. We refer the reader to the books \cite{D-G-Z} and \cite{FHHMPZ} for background on geometry of Banach spaces.

Let us finally recall the definition of Read's space and its basic properties from \cite{Read}. Let $c_{00}(\Q)$ be the set of all terminating sequences with rational coefficients, and let $(u_n)_{n \in \N}$ be a sequence of elements of $c_{00}(\Q)$ which lists every element infinitely many times.
Further, let  $(a_n)_{n \in \N}$ be a strictly increasing sequence of positive integers satisfying that
$$
a_n > \max \supp u_n \quad \text{and} \quad a_n> \|u_n\|_1
$$
for every $n \in \N$. The equivalent norm $|||\cdot|||$ on $c_0$ is defined in \cite{Read} as follows:
\begin{equation} \label{eq-norm1}
|||x||| := \|x\|_\infty + \sum_{n \in \N}2^{-a_n^2} \left|\langle x, u_n - e_{a_n} \rangle\right| \qquad \bigl(x\in c_0\bigr).
\end{equation}
Now, as we mentioned above, $\mathcal R := (c_0, |||\cdot|||)$. To simplify the notation, let us denote
$$
v_n = \frac{u_n - e_{a_n}}{\|u_n - e_{a_n}\|_1} \in \ell_1,\qquad  r_n = 2^{-a_n^2}\|u_n - e_{a_n}\|_1
$$
for every $n\in \N$. Then \eqref{eq-norm1} rewrites as
\begin{equation} \label{eq-norm2}
|||x||| = \|x\|_\infty + \sum_{n \in \N}r_n  \left|\langle x, v_n \rangle\right|
\end{equation}
for every $x\in \mathcal{R}$. We finally remark that $ \sum\limits_{n \in \N}r_n \leq 2$ \cite{Read} and that the sequence $(v_n)_{n\in N}$ is dense in $S_{\ell_1}$. Observe that it follows that
\begin{equation}\label{eq-equivalent}
\|x\|_\infty \leq |||x||| \leq 3 \|x\|_\infty
\end{equation}
for every $x\in \mathcal{R}$ \cite[Eq.~4]{Read}.

\section{The results}

Our first aim is to show that $\mathcal{R}^{**}$ is strictly convex but, previously, we need a description of $\mathcal{R}^{**}$ which can be of independent interest.

\begin{prop} \label{prop-bidualnorm}
The bidual space $\mathcal R^{**}$ of Read's space is naturally isometric to $\ell_\infty$ equipped with the norm given by the formula \eqref{eq-norm2}.
\end{prop}

\begin{proof}
Since $\mathcal R$ is $c_0$ in the equivalent norm $|||\cdot|||$, $\mathcal R^{**}$ should be  $\ell_\infty$ equipped with an equivalent norm $|||\cdot|||^{**}$. We want to demonstrate that  $(\ell_\infty, |||\cdot|||^{**}) = (\ell_\infty, |||\cdot|||)$. By Goldstine's theorem, the unit ball of  $\mathcal R^{**} = (\ell_\infty, |||\cdot|||^{**})$ is the weak$^*$-closure in $\ell_\infty$ of $B_{\mathcal R}$.  So, what remains to demonstrate is that the set $U:=\{\bar x \in \ell_\infty\colon |||\bar x||| \leq 1\}$ is equal to the  weak$^*$-closure in $\ell_\infty$ of $B_{\mathcal R}$. Recall that on bounded sets of  $\ell_\infty$,
the  weak$^*$-topology is metrizable (so  we can use the sequential language), and
the weak$^*$-convergence is just the coordinate-wise one. Let us demonstrate first that  $U$ is weak$^*$-closed, so $U$ will contain the weak$^*$-closure of $B_{\mathcal{R}}$, $B_\mathcal R^{**}$. Indeed, consider a sequence of $(\bar z_m)_{m\in \N}$ in $U$ with $w^*-\lim_m \bar z_m = \bar z \in \ell_\infty$. Since all the maps $\bar x \longmapsto \langle \bar x, v_n \rangle $ are weak$^*$-continuous on $\ell_\infty$, we have that $\lim_m \langle \bar z_m, v_n \rangle=  \langle \bar z, v_n \rangle$ for all $n \in \N$. Passing to a subsequence, we may (and do) assume that there exists $\lim_m  \|\bar z_m\|_\infty$ which satisfies that  $\lim_m  \|\bar z_m\|_\infty \geq \|\bar z\|_\infty$. Now,
\begin{align*}
|||\bar z||| &= \|\bar z\|_\infty + \sum_{n \in \N}r_n  \left|\langle \bar z, v_n \rangle \right| \leq   \lim_m\|\bar z_m\|_\infty + \sum_{n \in \N}r_n \lim_m  \left|\langle \bar  z_m, v_n \rangle\right|
\intertext{and using the version for series of Lebesgue's dominated convergence theorem}
&=  \lim_m\left(\|\bar z_m\|_\infty + \sum_{n \in \N}r_n   \left|\langle \bar z_m, v_n \rangle\right|\right) \leq 1,
\end{align*}
which demonstrates the desired  weak$^*$-closedness of $U$.

Let us show that $B_{\mathcal R}$ is weak$^*$-dense in $U$. Indeed, for every $\bar x = (x_1, x_2, \ldots) \in U \subset \ell_\infty$ and every $m\in \N$, denote $S_m\bar x = \sum_{k=1}^m x_ke_k \in c_0$. Then,  $w^*-\lim_m S_m \bar x = \bar x$ and, consequently, $\lim_m \langle S_m\bar x, v_n \rangle=  \langle \bar x, v_n \rangle$ for all $n \in \N$. Since $\lim_m \|S_m \bar x\|_\infty = \|\bar x\|_\infty$, another application of Lebesgue's dominated convergence theorem gives us
$$
 |||\bar x||| =  \lim_m  |||S_m\bar x|||.
$$
Now, consider $x_m = \frac{|||\bar x|||}{|||S_m\bar x|||}S_m\bar x$ for $m\in \N$. We have that $x_m  \in c_0$, $ |||x_m||| =  |||\bar x||| \leq 1$, so $x_m \in B_{\mathcal R}$ for every $m\in \N$. At the same time, $w^*-\lim_m x_m = w^*-\lim_m S_m \bar x = \bar x$. This completes the demonstration of the weak$^*$-density of $B_{\mathcal R}$ in $U$, and thus the proof of the theorem.
\end{proof}

From the proof above, we may extract the following property of Read's norm which we will use later and which is consequence of Lebesgue's dominated convergence theorem for series.

\begin{remark}\label{remark-Readnorm}
Let $(\bar z_m)_{m\in \N}$ be a sequence in $\ell_\infty$ which weakly$^*$-converges to $\bar z\in \ell_\infty$. Then, there exists $\lim_m |||\bar z_m|||$ if and only if there exists $\lim_m \|\bar z_m\|_\infty$. Besides,
$\lim_m |||\bar z_m|||=|||\bar z|||$ if and only if $\lim_m \|\bar z_m\|_\infty=\|\bar z\|_\infty$.
\end{remark}

Let us note that Proposition \ref{prop-bidualnorm} is a particular case of the following general result which is a consequence of the ``principle of local reflexivity for operators'' \cite{Behrends} and which has been suggested to us by the referee.

\begin{prop}\label{prop-bidual-abstract-version}
Let $X$, $Y$ be Banach spaces and let $R:X\longrightarrow Y$ be a bounded linear operator. Define a (equivalent) norm on $X$ by
$$
|||x|||=\|x\|_{X} + \|Rx\|_{Y}
$$
for every $x\in X$. Then the norm of $(X,|||\cdot|||)^{**}$ is given by the formula
$$
|||x^{**}|||=\|x^{**}\|_{X^{**}} + \|R^{**}x^{**}\|_{Y^{**}}
$$
for every $x^{**}\in X^{**}$ ($R^{**}$ denotes the biconjugate operator of $R$).
\end{prop}

\begin{proof}
Write
\begin{align*}
  B & =\{x\in X\colon \|x\|_X + \|Rx\|_Y\leq 1\}, \\
  A & =\{x^{**}\in X^{**}\colon \|x^{**}\|_{X^{**}} + \|R^{**}x^{**}\|_{Y^{**}}\leq 1\}, \\
  A_0 & =\{x^{**}\in X^{**}\colon \|x^{**}\|_{X^{**}} + \|R^{**}x^{**}\|_{Y^{**}}< 1\}.
\end{align*}
We have to prove that $A$ coincides with the weak$^*$-closure of $B$. It is clear that $A$ contains the weak$^*$-closure of $B$ since it contains $B$ and it is weak$^*$-closed by the weak$^*$ lower semicontinuity of the norm. Besides, it is enough to prove that the weak$^*$-closure of $B$ contains $A_0$. Therefore, we fix $x_0^{**}\in A_0$, we write $\delta=1-(\|x_0^{**}\|_{X^{**}} + \|R^{**}x_0^{**}\|_{Y^{**}})>0$, and we fix a weak$^*$ neighborhood $U$ of $x_0^{**}$ which we may suppose that is of the form
$$
U=\{z^{**}\in X^{**}\colon |\langle f, z^{**}-x_0^{**}\rangle|<\gamma \ \forall f\in D\}
$$
for suitable $D\subset S_{X^*}$ finite and $\gamma>0$. We can now apply the principle of local reflexivity for operators \cite[Theorem 5.2]{Behrends} to get a point $x_0\in X$ satisfying that
$$
\|x_0\|_{X}\leq \|x^{**}\|_{X^{**}}+\delta/2, \quad \|Rx_0\|_{Y}\leq \|R^{**}x_0^{**}\|_{Y^{**}}+\delta/2,\quad \text{and} \quad \langle x_0,f\rangle=\langle f,x_0^{**}\rangle\ \ \forall f\in D.
$$
We clearly have that $x_0\in B\cap U$, finishing the proof.
\end{proof}

We are now ready to present the strict convexity of the bidual of
$\mathcal{R}$. Recall that a Banach space $X$ is said to be
\emph{strictly convex} if $S_X$ does not contain any non-trivial
segment or, equivalently, if $\|x+y\|<2$ whenever $x,y\in B_X$,
$x\neq y$.

\begin{theorem} \label{theor-str-convR**}
The bidual space $\mathcal R^{**}$ of Read's space is strictly convex.
\end{theorem}

\begin{proof}
Let  $\bar x, \bar y \in S_{\mathcal R^{**}}$, $\bar x \neq \bar y$. Our goal is to demonstrate that $|||\bar x+\bar y||| < 2$. Let us use the fact that the sequence $(v_n)_{n\in N}$ from the formula \eqref{eq-norm2} is dense in $S_{\ell_1}$. This implies the existence of $k \in \N$ such that the values of $ \langle \bar x, v_k \rangle$ and $ \langle \bar y, v_k \rangle$ are non-zero and of opposite signs. Then,
$$
\left|\langle \bar x + \bar y, v_k \rangle \right| <   \left|\langle \bar x, v_k \rangle\right| + \left|\langle \bar y, v_k \rangle\right|.
$$
On the other hand, the triangle inequality says that
$$
\|\bar x+ \bar y\|_\infty  \leq \|\bar x\|_\infty + \|\bar y\|_\infty \qquad \text{and} \qquad  \left|\langle \bar x + \bar y, v_n \rangle \right| \leq  \left|\langle \bar x, v_n \rangle\right| + \left|\langle \bar y, v_n \rangle\right|
$$
for all $n \in \N\setminus \{k\}$ . Combining all these inequalities with the formula \eqref{eq-norm2}, we obtain the desired estimate:
\begin{align*}
|||\bar x + \bar y||| &= \|\bar x + \bar y\|_\infty + \sum_{n \in \N}r_n\left|\langle \bar x + \bar y, v_n \rangle\right| \\ &
<  \|\bar x\|_\infty + \sum_{n \in \N}r_n \left|\langle \bar x, v_n \rangle\right| +  \|\bar y\|_\infty + \sum_{n \in \N}r_n \left|\langle \bar y, v_n \rangle\right| = 2.\qedhere
\end{align*}
\end{proof}

As an immediate consequence, $\mathcal{R}^*$ is \emph{smooth}, i.e.\ its norm is G\^{a}teaux differentiable at every non-zero element (see \cite[Fact 8.12]{FHHMPZ}, for instance).

\begin{corollary}\label{corollary-dual-smooth}
The dual space $\mathcal{R}^*$ of Read's space is smooth.
\end{corollary}

However, the norm of $\mathcal{R}^*$ cannot be Fr\'{e}chet differentiable at any point, as we may show that it is rough. A norm of a Banach space is said to be $\eps$-\emph{rough} ($\eps>0$) if
$$
\limsup_{\|h\|\to 0}\frac{\|x+h\|+\|x-h\|-2\|x\|}{\|h\|}\geq \eps
$$
for every $x\in X$. The norm is said to be \emph{rough} if it is $\eps$-rough for some $\eps\in (0,2]$. We refer the reader to the classical book \cite{D-G-Z} on smoothness and renorming for more information and background. Clearly, a rough norm is not Fr\'{e}chet differentiable at any point.

\begin{theorem}\label{thm-rough}
The norm of the dual space $\mathcal{R}^*$ of Read's space is $2/3$-rough.
\end{theorem}

We need a formula for the supremum norm of $c_0$ which is well-known. It follows from the fact that $c_0$ has a property related to $M$-ideals called $(m_\infty)$ and which was deeply studied by N~.Kalton and D.~Werner in \cite{KaltonWerner}. We include a simple proof of a slightly more general result here for the sake of completeness.

\begin{lemma}\label{conorm}
Assume that $(u_m)_{m\in \N}$ is a weakly$^*$-null sequence in
$\ell_\infty$ such that the limit $\lim_m\Vert u_m\Vert_{\infty}$ exists.
Then, for every $u\in c_0$
$$
\lim_m \Vert u+u_m\Vert_{\infty} =\max\bigl\{\Vert u\Vert_{\infty},\ \lim_{m}\Vert
u_m\Vert_{\infty}\bigr\}.
$$
\end{lemma}

\begin{proof} Suppose first that $u$ has finite support. Let $n\in\N$ be an integer such that all non-zero coordinates of $u$ are smaller than $n$ and denote by $P$ the projection of $\ell_\infty$ on the first $n$ coordinates. Then, $(Pu_m)_{m\in\N}$ go coordinate-wise to zero, but there are only $n$ non-null coordinates, so actually $(Pu_m)_{m\in \N}$ go to zero in norm. With that in mind, we get that
\begin{align*}
\lim_m\|u + u_m\|_\infty & = \lim_m\|u + (u_m - Pu_m)\|_\infty =  \lim_m \max\bigl\{\|u\|_\infty,\, \|u_m - Pu_m\|_\infty\bigr\} \\ &=  \max\bigl\{\|u\|_\infty,\,  \lim_m\|u_m\|_\infty\bigr\},
\end{align*}
where in the second equality we have used the disjointness of the supports of $u$ and $u_m - Pu_m$.

Now, let us define $f_m:c_0\longrightarrow \R$ by $f_m(u)=\|u+u_m\|_\infty$ for every $u\in c_0$ and every $m\in \N$. Then, all the functions $f_m$ are $1$-Lipschitz and the sequence $(f_m)_{m\in \N}$ converges pointwise on the dense set $c_{00}$ to a function which is a fortiori $1$-Lipschitz. It is now routine, using again that the Lipschitz constants of the $f_m$'s are uniformly bounded, to prove that $(f_m)_{m\in \N}$ converges pointwise on the whole $c_0$ to the unique extension of the limit on $c_{00}$.
\end{proof}

We are now ready to proof that the norm of $\mathcal{R}^*$ is rough.

\begin{proof}[Proof of Theorem~\ref{thm-rough}]
Fix $x_0^*\in S_{\mathcal{R}^*}$ and $\lambda\in (0,1)$. Given $\delta\in (0,1/3)$, we take $x_0\in B_{\mathcal{R}^*}$ such that
$$
x_0^*(x_0)>1-3\lambda\delta \qquad \text{and} \qquad |||x_0|||<1-\lambda\delta.
$$
Write $\rho=1/3-\lambda\delta$. We have that $\|x_0\|_\infty\geq \frac13 |||x_0|||\geq \rho$ (use \eqref{eq-equivalent}) and that the sequence $(\rho\, e_m)_{m\in \N}$ is weakly-null, so Lemma \ref{conorm} gives us that
$$
\lim_m \|x_0 \pm \rho\, e_m\|_\infty = \max\bigl\{\|x_0\|_\infty,\ \lim_m \|\rho\, e_m\|_\infty\bigr\}=\|x_0\|_\infty.
$$
It then follows from Remark~\ref{remark-Readnorm} that
$$
\lim_m |||x_0 \pm \rho\, e_m||| = |||x_0||| < 1-\lambda\delta.
$$
With all of these in mind, we may find $N\in \N$ such that
\begin{equation}\label{eq:rough1}
|||x_0 \pm \rho\, e_N|||\leq 1 \quad \text{and} \quad x_0^*(x_0\pm \rho \, e_N)>1-3\lambda\delta.
\end{equation}
Finally, take $y^*\in S_{\mathcal{R}^*}$ such that
$$
y^*(e_N)=|||e_N|||\geq 1.
$$
We have that
\begin{align*}
  |||x_0^* + \lambda y^*||| + |||x_0^*-\lambda y^*||| &\geq \langle x_0 + \rho\, e_N, x_0^*+\lambda y^*\rangle + \langle x_0 - \rho\, e_N, x_0^*-\lambda y^*\rangle \\
   & > 2 - 6\lambda\delta + 2\lambda\rho = 2 +\lambda\left(\frac{2}{3}-(6+\lambda)\delta \right).
\end{align*}
Summarizing, we have proved that for every $x^*\in S_{\mathcal{R}^*}$ and every $\lambda\in (0,1)$,
$$
\sup_{z^*\in \mathcal{R},\,|||z^*|||=\lambda}\frac{|||x^*+z^*||| + |||x^*-z^*|||-2}{\lambda} \geq \frac{2}{3}.
$$
This gives the $2/3$-roughness of $\mathcal{R}^*$, as desired.
\end{proof}

Observe that following the above proof until \eqref{eq:rough1}, we get that every slice of the unit ball of $\mathcal{R}$ has diameter greater than or equal to $2/3$. Actually, the $\eps$-roughness of the norm of the dual $X^*$ of a Banach space $X$ is equivalent to the fact that all slices of the unit ball of $X$ have diameter greater than or equal to $\eps$ \cite[Proposition~I.1.11]{D-G-Z}, and the second part of our proof is based in the proof of the result above. Let us also note that the first part of the proof of Theorem \ref{thm-rough} can be easily adapted to get that all weakly open subsets of $B_{\mathcal{R}}$ have diameter greater than or equal to $2/3$.

\begin{corollary}\label{Corollary-big-weak-open-subsets}
Every weakly open subset of the unit ball of Read's space $\mathcal{R}$ has diameter greater than or equal to $2/3$.
\end{corollary}

We now study convexity properties of Read's space itself. It follows from Theorem \ref{theor-str-convR**} that $\mathcal{R}$ is strictly convex (as it is a subspace of $\mathcal{R}^{**}$), but we may actually prove that it
is weakly locally uniformly rotund. Recall that a Banach space $X$
is \emph{weakly locally uniformly rotund} (\emph{WLUR} in short)
if for every $x\in S_X$ and every sequence $(x_n)_{n\in \N}$ in
$S_X$, if $\|x+x_n\|\longrightarrow 2$, then $x=w-\lim_n x_n$. If one actually
gets that $x=\lim_n x_n$ in norm, we say that the space $X$ is
\emph{locally uniformly rotund} (\emph{LUR} in short). It is clear
that LUR implies WLUR and that WLUR implies strict convexity,
being the converse results false.

\begin{theorem} \label{Theo-Read-WLUR}
Read's space $\mathcal R$  is weakly locally uniformly rotund.
\end{theorem}

\begin{proof}
Let $x, y_m \in S_{\mathcal R}$ with $|||x + y_m||| \longrightarrow 2$. Observe that it is enough to show that there is subsequence of $(y_m)_{m\in \N}$ which weakly converges to $x$. Passing to a subsequence, we may assume the existence of  $w^*-\lim_m y_m = \bar y \in B_{\mathcal R^{**}}$. Then, using Remark~\ref{remark-Readnorm},
\begin{equation}  \label{eq00000}
\begin{split}
 2  = \lim_m |||x+y_m||| &= \lim_m \|x + y_m\|_\infty + \sum_{n \in \N}r_n \left|\langle x + \bar y, v_n \rangle \right|  \\ & \leq  \|x \|_\infty  + \lim_m \|y_m\|_\infty                                        + \sum_{n \in \N}r_n  \left|\langle x, v_n \rangle \right|+ \sum_{n \in \N}r_n  \left|\langle \bar y, v_n \rangle \right| \\ &= |||x||| + \lim_m |||y_m||| = 2.
\end{split}
\end{equation}
This chain of inequalities implies that
\begin{equation} \label{eq-theo2-2}
\lim_m \|x + y_m\|_\infty  =  \|x \|_\infty  + \lim_m \|y_m\|_\infty
\end{equation}
and
\begin{equation} \label{eq-theo2-3}
 \left|\langle x + \bar y, v_n \rangle \right|  =  \left|\langle x, v_n \rangle \right|+ \left|\langle \bar y, v_n \rangle \right| \quad \textrm{ for all \,} n \in \N.
\end{equation}
As in the proof of Theorem \ref{theor-str-convR**}, if $\bar y$ is not of the form $ax$ with $a \geq 0$,  there exists $k \in \N$ such that the values of $ \langle x, v_k \rangle$ and $ \langle \bar y, v_k \rangle$ are non-zero and of opposite signs. Then $ \left|\langle x + \bar y, v_k \rangle \right| <   \left|\langle x, v_k \rangle\right| + \left|\langle \bar y, v_k \rangle\right|$, which contradicts \eqref{eq-theo2-3}. So,  $\bar y = ax \in \mathcal R$ for some $a \in [0,  +\infty)$. Since $1 =  \lim_m |||y_m||| \geq |||ax||| = a$, we deduce that $a \in [0, 1]$.

Passing again to a subsequence, we may assume that there exists $ \lim_m \|y_m - ax\|_\infty$. From Lemma \ref{conorm} we obtain that
\begin{equation} \label{eq-theo2-3+}
 \lim_m \|y_m\|_\infty = \max\left\{a\|x\|_\infty,\ \lim_m \|y_m - ax\|_\infty\right\},
\end{equation}
and
$$
\lim_m \|x + y_m\|_\infty  =  \max\left\{(1+a)\|x\|_\infty,\ \lim_m \|y_m - ax\|_\infty\right\}.
$$
Combining this with \eqref{eq-theo2-2}, we obtain that
$$
  \|x \|_\infty  +  \max\left\{a\|x\|_\infty,\ \lim_m \|y_m - ax\|_\infty\right\} =  \max\left\{(1+a)\|x\|_\infty,\ \lim_m \|y_m - ax\|_\infty\right\},
$$
which implies that
$$
 \lim_m \|y_m - ax\|_\infty \leq a\|x\|_\infty,
$$
so, by \eqref{eq-theo2-3+},
$$
 \lim_m \|y_m\|_\infty =a\|x\|_\infty.
$$
Taking into account that all the inequalities in  \eqref{eq00000} are, in fact, equalities, and substituting $\bar y = ax$,
we get
$$
 2  =  \|x \|_\infty  + \lim_m \|y_m\|_\infty+ \sum_{n \in \N}r_n  \left|\langle x, v_n \rangle \right|+ \sum_{n \in \N}r_n  \left|\langle ax, v_n \rangle \right| = (1 + a)|||x||| = 1+a.
$$
So $a = 1$ and $w-\lim_m(y_m - x) = ax - x = 0$.
\end{proof}

The above result cannot be improved to get that $\mathcal{R}$ is LUR. This is so because it follows easily from the definition, that the unit ball of a LUR space contains slices or arbitrarily small diameter and this contradicts Corollary \ref{Corollary-big-weak-open-subsets} or, alternatively, because the norm of the dual space of a LUR space is Fr\'{e}chet differentiable at every norm-attaining functional by the Smulyan's criterium (see \cite[Theorem I.1.4]{D-G-Z}, for instance).

\begin{remark}
Read's space $\mathcal R$ is not locally uniformly rotund.
\end{remark}

Observe that $\mathcal R$ is not smooth (this follows from the formula for the directional derivative of the norm $|||\cdot|||$ given in \cite[Lemma 2.5]{Read}). Nevertheless, one can modify  $\mathcal R$ in such a way that the modified space $\widetilde{\mathcal R}$  is simultaneously strictly convex and smooth (actually, its dual is also smooth and strictly convex), but the set of norm-attaining functionals remains the same. This is a consequence of the following argument which appears in the proof of \cite[Theorem 9.(4)]{DebsGodSR} by G.~Debs, G.~Godefroy, and J.\ Saint Raymond, and which we include here for the sake of completeness.

\begin{lemma} \label{lemma-smooth-renorming}
Let $X$ be a separable Banach space, let $\{x_n\colon n\in \N\}$ be a dense subset of $B_X$, and consider the bounded linear operator $T: \ell_2 \longrightarrow X$ be defined by $T\bigl((a_n)_{n\in \N}\bigr)=\sum\limits_{n=1}^{+\infty} \frac{a_n}{2^n} x_n$ for every $(a_n)_{n\in \N}\in \ell_2$. Consider the equivalent norm on $X$, denoted by $\|\cdot\|_{s}$, for which the set $V = B_X + T(B_{\ell_2})$ is its unit ball. Then, $(X,\|\cdot\|_{s})^*$ is strictly convex and so, $(X,\|\cdot\|_{s})$ is smooth, and $\NA(X) = \NA(X,\|\cdot\|_{s})$. Moreover, if $X$ is strictly convex, then  $(X,\|\cdot\|_{s})$ is also strictly convex; if $X^*$ is smooth, then $(X,\|\cdot\|_{s})^*$ is also smooth.
\end{lemma}

\begin{proof}
The set $V = B_X + T(B_{\ell_2})$ is bounded, balanced and solid. Its closedness follows from the compactness of  $T(B_{\ell_2})$. This explains the definition of $\|\cdot\|_{s}$. A functional $f \in X^*$ attains its maximum on $V$ if and only if it attains its maximum both on $B_X$ and $T(B_{\ell_2})$, but all functionals attain their maximums on $T(B_{\ell_2})$, so $\NA(X) = \NA(X,\|\cdot\|_{s})$. It is easy to show that for every $f\in X^*$, $\|f\|_{s}=\|f\|+\|T^*(f)\|_2$. Since $T^*$ is one-to-one (as $T$ has dense range) and $\|\cdot\|_2$ is strictly convex, it follows that $(X,\|\cdot\|_{s})^*$ is strictly convex and so $(X,\|\cdot\|_{s})$ is smooth.

Finally, if $X$ is strictly convex, so is $(X,\|\cdot\|_{s})$ as a functional $f \in X^*$ cannot attain its maximum in two different points of $V$; if $X^*$ is smooth, then so is $(X,\|\cdot\|_{s})^*$ as its norm is the sum of two smooth norms.
\end{proof}

We are now ready to present a smooth version of Read's space.

\begin{example}
{\slshape Consider $\widetilde{\mathcal{R}}$ to be the renorming of Read's space $\mathcal{R}$ given by the procedure of the above lemma. Then, $\widetilde{\mathcal{R}}^*$ is strictly convex and smooth, so $\widetilde{\mathcal{R}}$ is also strictly convex and smooth, and solves negatively both problems (S) and (G). That is, $\widetilde{R}$ does not contains proximinal subspaces of finite codimension greater than or equal to 2 and $\NA(\widetilde{\mathcal{R}})$ does not contain two-dimensional linear subspaces.}

Indeed, $\widetilde{\mathcal{R}}$ and $\widetilde{\mathcal{R}}^*$ are smooth, and $\NA(\widetilde{\mathcal R})= \NA(\mathcal R)$, so $\NA(\widetilde{\mathcal R})$ does not contain linear subspaces of dimension greater than or equal to 2. But then it is easy to show that this implies that $\widetilde{\mathcal{R}}$ does not contain proximinal subspaces of finite codimension greater than or equal to 2 (see \cite[Proposition III.4]{Godefroy}, for instance).
\end{example}

\vspace*{1cm}

\noindent \textbf{Acknowledgment:\ } The authors are grateful to the anonymous referee for helpful suggestions which improved the final version of the paper. In particular, Proposition \ref{prop-bidual-abstract-version} and the possibility of Theorem \ref{thm-rough} to be true were suggested by the referee.

\end{document}